\documentclass[11pt]{amsart}

\author[C.~Sanna]{Carlo Sanna$^\dagger$}
\address{\parbox{\linewidth}{
Politecnico di Torino, Department of Mathematical Sciences\\
Corso Duca degli Abruzzi 24, 10129 Torino, Italy\\[-8pt]}}
\email{carlo.sanna.dev@gmail.com}
\thanks{$\dagger\,$C.~Sanna is a member of GNSAGA of INdAM and of CrypTO, the group of Cryptography and Number~Theory of Politecnico di Torino}

\title{On the l.c.m.~of shifted Fibonacci numbers}

\usepackage{amsmath}
\usepackage{amssymb}
\usepackage{amsthm}
\usepackage{geometry}
\usepackage{color}
\usepackage{hyperref}
\usepackage{enumerate}
\hypersetup{colorlinks=true}
\usepackage{placeins}

\newtheorem{thm}{Theorem}[section]

\newtheorem{lem}[thm]{Lemma}
\theoremstyle{remark}

\DeclareMathOperator*{\lcm}{lcm}
\DeclareMathOperator{\Li}{Li}

\begin{document}

\begin{abstract}
Let $(F_n)_{n \geq 1}$ be the sequence of Fibonacci numbers.
Guy and Matiyasevich proved that 
\begin{equation*}
\log \lcm (F_1, F_2, \dots, F_n) \sim \frac{3 \log \alpha}{\pi^2} \cdot n^2 \quad \text{as } n \to +\infty,
\end{equation*}
where $\lcm$ is the least common multiple and $\alpha := \big(1 + \sqrt{5}) / 2$ is the golden~ratio.

We prove that for every periodic sequence $\mathbf{s} = (s_n)_{n \geq 1}$ in $\{-1,+1\}$ there exists an effectively computable rational number $C_{\mathbf{s}} > 0$ such that
\begin{equation*}
\log \lcm (F_3 + s_3, F_4 + s_4, \dots, F_n + s_n) \sim \frac{3 \log \alpha}{\pi^2} \cdot C_\mathbf{s} \cdot n^2 , \quad \text{as } n \to +\infty .
\end{equation*}
Moreover, we show that if $(s_n)_{n \geq 1}$ is a sequence of independent uniformly distributed random variables in $\{-1,+1\}$ then
\begin{equation*}
\mathbb{E}\big[\log \lcm (F_3 + s_3, F_4 + s_4, \dots, F_n + s_n)\big] \sim \frac{3 \log \alpha}{\pi^2} \cdot \frac{15 \Li_2(1 / 16)}{2} \cdot n^2 , \quad \text{as } n \to +\infty ,
\end{equation*}
where $\Li_2$ is the dilogarithm function.
\end{abstract}

\maketitle

\section{Introduction}

Let $(F_n)_{n \geq 1}$ be the sequence of Fibonacci numbers, defined recursively by $F_1 = F_2 = 1$ and $F_{n + 2} = F_{n + 1} + F_n$, for every integer $n \geq 1$.
Guy and Matiyasevich~\cite{MR1712797} proved that, as $n \to +\infty$, 
\begin{equation}\label{equ:guy}
\log \lcm (F_1, F_2, \dots, F_n) \sim \frac{3 \log \alpha}{\pi^2} \cdot n^2 ,
\end{equation}
where $\lcm$ denotes the least common multiple and $\alpha := \big(1 + \sqrt{5}) / 2$ is the golden~ratio.
This result was extended by Kiss--M\'{a}ty\'{a}s~\cite{MR993902}, Akiyama~\cite{MR1077711}, and Tropak~\cite{MR1114366} to more general binary recurrences, and by Akiyama~\cite{MR1242715, MR1394375} to sequences satisfying some special divisibility properties (see also~\cite{MR3150887}).

We study what happens if each Fibonacci number $F_k$ in~\eqref{equ:guy} is replaced by a \emph{shifted} Fibonacci number $F_k \pm 1$, for various choices of signs.
Arithmetic properties of shifted Fibonacci have been studied before.
For example, Bugeaud, Luca, Mignotte, and Siksek~\cite{MR2399185} determined all the shifted Fibonacci numbers that are perfect powers; Marques~\cite{MR2892008} gave formulas for the order of appearance of shifted Fibonacci numbers; and Pongsriiam~\cite{MR3620575} found all shifted Fibonacci numbers that are products of Fibonacci numbers.

Our first result concerns periodic sequences of signs.

\begin{thm}\label{thm:periodic}
For every periodic sequence $\mathbf{s} = (s_n)_{n \geq 1}$ in $\{-1, +1\}$, there exists an effectively computable rational number $C_\mathbf{s} > 0$ such that
\begin{equation*}
\log \lcm (F_3 + s_3, F_4 + s_4, \dots, F_n + s_n) \sim \frac{3 \log \alpha}{\pi^2} \cdot C_\mathbf{s} \cdot n^2 ,
\end{equation*}
as $n \to +\infty$. 
(The least common multiple starts from $F_3 + s_3$ to avoid zero terms.)
\end{thm}

We computed the constant $C_{\mathbf{s}}$ for periodic sequences $\mathbf{s}$ with short period.
We found that $C_{\mathbf{s}} = 1 / 2$ for most of such sequences.
In particular, $C_{\mathbf{s}} = 1 / 2$ for all periodic sequences with period less than $5$.
Moreover, all the periodic sequences $\mathbf{s}$ with $C_{\mathbf{s}} \neq 1 / 2$ and period $5$ or $6$ are listed in Table~\ref{tab:period5} and Table~\ref{tab:period6}, respectively.

\begin{table}[h]
\centering
\begin{tabular}{cc|cc|cc|cc}
$\mathbf{s}$ & $C_{\mathbf{s}}$ & $\mathbf{s}$ & $C_{\mathbf{s}}$ & $\mathbf{s}$ & $C_{\mathbf{s}}$ & $\mathbf{s}$ & $C_{\mathbf{s}}$ \\\hline
\texttt{----+} & $43/96$ & \texttt{-+--+} & $43/96$   & \texttt{+----} & $43/96$  & \texttt{+-+++} & $91/192$ \rule{0pt}{2.6ex}\\
\texttt{---+-} & $43/96$ & \texttt{-+-+-} & $43/96$   & \texttt{+--+-} & $43/96$  & \texttt{++-+-} & $17/36$  \\
\texttt{--+--} & $11/24$ & \texttt{-+-++} & $91/192$  & \texttt{+-+--} & $43/96$  & \texttt{++-++} & $17/36$  \\
\texttt{--+-+} & $11/24$ & \texttt{-++-+} & $91/192$  & \texttt{+-+-+} & $91/192$ & \texttt{+++-+} & $91/192$ \\
\texttt{-+---} & $43/96$ & \texttt{-++++} & $91/192$  & \texttt{+-++-} & $91/192$ & \texttt{++++-} & $91/192$ \\
\end{tabular}
\vspace{1em}
\caption{All period-$5$ sequences $\mathbf{s}$ such that $C_{\mathbf{s}} \neq 1/2$.}
\label{tab:period5}
\end{table}

\begin{table}[h]
\centering
\begin{tabular}{cc|cc|cc|cc}
$\mathbf{s}$ & $C_{\mathbf{s}}$ & $\mathbf{s}$ & $C_{\mathbf{s}}$ & $\mathbf{s}$ & $C_{\mathbf{s}}$ & $\mathbf{s}$ & $C_{\mathbf{s}}$ \\\hline
\texttt{-----+} & $13/32$ & \texttt{--++-+} & $7/16$  & \texttt{+-+-++} & $29/64$ & \texttt{++--++} & $11/24$ \rule{0pt}{2.6ex}\\
\texttt{----++} & $13/32$ & \texttt{--++++} & $29/64$ & \texttt{+-++--} & $29/64$ & \texttt{++-+--} & $13/32$ \\
\texttt{---+--} & $7/16$  & \texttt{-+----} & $13/32$ & \texttt{+-+++-} & $29/64$ & \texttt{+++-+-} & $11/24$ \\
\texttt{---+-+} & $7/16$  & \texttt{-+---+} & $13/32$ & \texttt{+-++++} & $29/64$ & \texttt{+++-++} & $11/24$ \\
\texttt{--+-++} & $29/64$ & \texttt{-+--++} & $13/32$ & \texttt{++----} & $13/32$ & \texttt{++++--} & $29/64$ \\
\texttt{--++--} & $7/16$  & \texttt{-+-+--} & $13/32$ & \texttt{++--+-} & $11/24$ & \texttt{+++++-} & $29/64$ \\
\end{tabular}
\vspace{1em}
\caption{All period-$6$ sequences $\mathbf{s}$ such that $C_{\mathbf{s}} \neq 1/2$.}
\label{tab:period6}
\end{table}

\FloatBarrier

Our second result regards random sequences of signs. 

\begin{thm}\label{thm:random}
Let $(s_n)_{n \geq 1}$ be a sequence of independently uniformly distributed random variables in $\{-1, +1\}$.
Then
\begin{equation*}
\mathbb{E}\big[\log \lcm (F_3 + s_3, F_4 + s_4, \dots, F_n + s_n)\big] \sim \frac{3 \log \alpha}{\pi^2} \cdot \frac{15 \Li_2(1 / 16)}{2} \cdot n^2 ,
\end{equation*}
as $n \to +\infty$, where $\Li_2(z) := \sum_{n\,=\,1}^\infty z^n / n^2$ denotes the dilogarithm.
\end{thm}

Using the methods of the proofs of Theorem~\ref{thm:periodic} and Theorem~\ref{thm:random}, it should be possible to prove similar results, where the sequence of Fibonacci numbers is replaced by the sequence of Lucas numbers or by a sequence of integers powers $(a^n)_{n \geq 1}$, with $a \geq 2$ a fixed integer.
Also, one could consider what happens for a deterministic non-periodic sequence of signs $(s_n)_{n \geq 1}$.
We leave these as problems for the interested reader.

\subsection*{Notation}

We employ the Landau–Bachmann ``Big Oh'' notation $O$ with its usual meaning.
Any dependence of the implied constants is indicated with subscripts.
We let $\varphi$ denote the Euler's totient function.
We reserve the letter $p$ for prime numbers.

\section{Preliminaries on Fibonacci and Lucas Numbers}

Let $(L_n)_{n \geq 1}$ be the sequence of Lucas numbers, defined recursively by $L_1 = 1$, $L_2 = 3$, and $L_{n + 2} = L_{n + 1} + L_n$, for every integer $n \geq 1$.
It is well known that the Binet's formulas
\begin{equation}\label{equ:binet}
F_n = \frac{\alpha^n - \beta^n}{\alpha - \beta} \qquad\text{and}\qquad L_n = \alpha^n + \beta^n ,
\end{equation}
hold for every integer $n \geq 1$, where $\alpha := \big(1 + \sqrt{5}) / 2$ and $\beta := \big(1 - \sqrt{5}) / 2$.
It is useful (proof of Lemma~\ref{lem:shifted} later) to extend the sequences of Fibonacci and Lucas numbers to negative indices using~\eqref{equ:binet}.
Let us define
\begin{equation}\label{equ:cyclotomic}
\Phi_n := \prod_{\substack{1 \,\leq\, k \,\leq\, n \\ \gcd(n, k) \,=\, 1}} \left(\alpha - \mathrm{e}^{\frac{2 \pi \mathbf{i} k}{n}} \beta \right) ,
\end{equation}
for each integer $n \geq 2$, and put $\Phi_1 := 1$.
It can be proved that each $\Phi_n$ is an integer~\cite[p.~428]{MR491445}.
Moreover, from \eqref{equ:binet} and \eqref{equ:cyclotomic} it follows that
\begin{equation}\label{equ:FnLnPhid}
F_n = \prod_{n \,\in\, \mathcal{D}(n)} \Phi_d \qquad\text{and}\qquad L_n = \prod_{n \,\in\, \mathcal{D}^\prime(n)} \Phi_d ,
\end{equation}
for every integer $n \geq 1$, where $\mathcal{D}(n) := \{d \in \mathbb{N} : d \mid n\}$ and $\mathcal{D}^\prime(n) := \mathcal{D}(2n) \setminus \mathcal{D}(n)$.
In~particular, using~\eqref{equ:FnLnPhid} one can prove by induction that $\Phi_n > 0$ for every integer $n \geq 1$.

We need the following two results about $\Phi_n$.

\begin{lem}\label{lem:gcdPhi}
For all integers $m > n \geq 1$ we have $\gcd(\Phi_m, \Phi_n) \mid m$.
\end{lem}
\begin{proof}
For $m \geq 5$, $m \neq 6, 12$, and $n \geq 3$, it is known~\cite[Lemma~7]{MR491445} that $\gcd(\Phi_m, \Phi_n)$ divides the greatest prime factor of $m /\!\gcd(3, m)$, and consequently it divides $m$.
The remaining cases follow easily since $\Phi_1 = \Phi_2 = 1$, $\Phi_3 = 2$, $\Phi_4 = 3$, $\Phi_5 = 5$, $\Phi_6 = 4$, and $\Phi_{12} = 6$.
\end{proof}

\begin{lem}\label{lem:logPhi}
For all integers $n \geq 1$, we have $\log \Phi_n = \varphi(n) \log \alpha + O(1)$.
\end{lem}
\begin{proof}
See, e.g.,~\cite[Lemma~2.1(iii)]{MR4003803}.
\end{proof}

The next lemma belongs to the folklore and provides a way to write shifted Fibonacci numbers as products of Fibonacci and Lucas numbers.

\begin{lem}\label{lem:shifted}
For every integer $k$, we have
\begin{alignat*}{4}
F_{4k + 1} - 1 &= F_{2k} L_{2k + 1},     \qquad & F_{4k + 1} + 1 &= F_{2k + 1} L_{2k},     \\
F_{4k + 2} - 1 &= F_{2k} L_{2k + 2},     \qquad & F_{4k + 2} + 1 &= F_{2k + 2} L_{2k},     \\
F_{4k + 3} - 1 &= F_{2k + 2} L_{2k + 1}, \qquad & F_{4k + 3} + 1 &= F_{2k + 1} L_{2k + 2}, \\
F_{4k + 4} - 1 &= F_{2k + 3} L_{2k + 1}, \qquad & F_{4k + 4} + 1 &= F_{2k + 1} L_{2k + 3}.
\end{alignat*}
\end{lem}
\begin{proof}
Employing~\eqref{equ:binet} and $\alpha \beta = -1$, a quick algebraic manipulation yields
\begin{equation}\label{equ:sumdiff}
F_{a + b} + (-1)^b F_{a - b} = F_a L_b ,
\end{equation}
for all integers $a, b$.
Each of the eight identities corresponds to a particular choice of~$a,b$ in~\eqref{equ:sumdiff}, noting that $F_{-1} = 1$ and $F_{-2} = -1$.
\end{proof}

Finally, we need a lemma about the greatest common divisor of a Fibonacci number and a Lucas number.

\begin{lem}\label{lem:gcdFmLn}
For all integers $m, n$, we have that $\gcd(F_m, L_n)$ is equal to $1$, $2$, or $L_{\gcd(m, n)}$.
\end{lem}
\begin{proof}
See~\cite{MR1089516}.
\end{proof}

\section{Further preliminaries}

For every sequence $\mathbf{s} = (s_n)_{n \geq 1}$ in $\{-1, +1\}$ and for every integer $n \geq 5$, define
\begin{equation*}
\ell_{\mathbf{s}}(n) = \lcm(F_5 + s_5, \dots, F_n + s_n) .
\end{equation*}
(Starting from $F_5$ instead of $F_3$ does not affect the asymptotic and simplifies a bit the next arguments.)
Furthermore, define the sets
\begin{alignat*}{3}
\mathcal{F}_{\mathbf{s}}(n) := \big\{h \in [2,\tfrac{n}{2}] :\phantom{M} s_{2h - 2} &= (-1)^h &\;\lor\; s_{2h - 1} &= (-1)^{h+1} \\
 \;\lor\; s_{2h + 1} &= (-1)^{h+1} &\;\lor\; s_{2h + 2} &= (-1)^{h+1} \big\} , \\[0.5em]
\mathcal{L}_{\mathbf{s}}(n) := \big\{h \in [2,\tfrac{n}{2}] :\phantom{M} s_{2h - 2} &= (-1)^{h+1} &\;\lor\; s_{2h - 1} &= (-1)^h \\
 \;\lor\; s_{2h + 1} &= (-1)^h &\;\lor\; s_{2h + 2} &= (-1)^h \big\} ,
\end{alignat*}
and
\begin{equation*}
\mathcal{M}_{\mathbf{s}}(n) := \bigcup_{h \,\in\, \mathcal{F}_{\mathbf{s}}(n)} \mathcal{D}(h) \;\cup\; \bigcup_{h \,\in\, \mathcal{L}_{\mathbf{s}}(n)} \mathcal{D}^\prime(h) .
\end{equation*}
The next lemma is the key to the proofs of Theorem~\ref{thm:periodic} and Theorem~\ref{thm:random}.

\begin{lem}\label{lem:lcm}
As $n \to +\infty$, we have
\begin{equation*}
\log \ell_{\mathbf{s}}(n) = \sum_{d \,\in\, \mathcal{M}_{\mathbf{s}}(n)} \varphi(d) \log \alpha + O\!\left(\frac{n^2}{\log n}\right) .
\end{equation*}
\end{lem}
\begin{proof}
Assume $n \geq 8$ and let $n = 4K + 4$, for some real number $K \geq 1$.
Using Lemma~\ref{lem:shifted}, we can write each $F_i + s_i$ ($i=5,\dots,n$) as a product of a Fibonacci number and a Lucas number, which, in light of Lemma~\ref{lem:gcdFmLn}, have a greatest common divisor not exceeding $3$.
Therefore,
\begin{equation}\label{equ:logelln1}
\log \ell_{\mathbf{s}}(n) = \log\lcm\!\left(\lcm_{i \,\in\, \mathcal{F}_{\mathbf{s}}^\prime(n)} F_i, \lcm_{j \,\in\, \mathcal{L}_{\mathbf{s}}^\prime(n)} L_j \right) + O(1) ,
\end{equation}
where $\mathcal{F}_{\mathbf{s}}^\prime(n), \mathcal{L}_{\mathbf{s}}^\prime(n) \subseteq [2, 2K + 3] \cap \mathbb{Z}$ are defined by
\begin{align}\label{equ:calFcalL}
2k \in \mathcal{F}_{\mathbf{s}}^\prime(n) \quad\Longleftrightarrow\quad &\big((1 \leq k \leq K) \;\land\; (s_{4k + 1} = -1 \;\lor\; s_{4k + 2} = -1) \big) \\
 &\phantom{M} \lor\; \big((2 \leq k \leq K + 1) \;\land\; (s_{4k - 1} = - 1 \;\lor\; s_{4k - 2} = +1)\big) , \nonumber\\[0.3em]
2k + 1 \in \mathcal{F}_{\mathbf{s}}^\prime(n) \quad\Longleftrightarrow\quad &\big((1 \leq k \leq K) \;\land\; (s_{4k + 1} = +1 \;\lor\; s_{4k + 3} = +1 \;\lor\; s_{4k + 4} = +1)\big) \nonumber\\
 &\phantom{M} \lor\; \big((2 \leq k \leq K + 1) \;\land\; s_{4k} = - 1\big) , \nonumber\\[0.3em]
2k \in \mathcal{L}_{\mathbf{s}}^\prime(n) \quad\Longleftrightarrow\quad &\big((1 \leq k \leq K) \;\land\; (s_{4k+1} = +1 \;\lor\; s_{4k + 2} = + 1)\big) \nonumber\\
 &\phantom{M} \lor\; \big((2 \leq k \leq K + 1) \;\land\; (s_{4k - 2} = - 1 \;\lor\; s_{4k - 1} = +1)\big) , \nonumber\\[0.3em]
2k + 1 \in \mathcal{L}_{\mathbf{s}}^\prime(n) \quad\Longleftrightarrow\quad &\big((1 \leq k \leq K) \;\land\; (s_{4k + 1} = -1 \;\lor\; s_{4k + 3} = -1 \;\lor\; s_{4k + 4} = -1)\big) \nonumber\\
 &\phantom{M} \lor\; \big((2 \leq k \leq K + 1) \;\land\; s_{4k} = +1\big) , \nonumber
\end{align}
for every integer $k \in [1, K + 1]$.
Since $F_i, L_i \leq 2^i$ for every integer $i \geq 1$, replacing all the bounds on $k$ in~\eqref{equ:calFcalL} with $2 \leq k \leq n / 4$ amount to an error at most $O(n)$ in~\eqref{equ:logelln1}, that is,
\begin{equation}\label{equ:logelln2}
\log \ell_{\mathbf{s}}(n) = \log\lcm\!\left(\lcm_{i \,\in\, \mathcal{F}_{\mathbf{s}}(n)} F_i, \lcm_{j \,\in\, \mathcal{L}_{\mathbf{s}}(n)} L_j \right) + O(n) .
\end{equation}

Suppose that $p^v \mid\mid \ell_\mathbf{s}(n)$, for some prime number $p \leq n$ and some integer $v \geq 1$.
Then $p^v \mid F_i + s_i$ for some integer $i \in [5, n]$, and consequently $p^v \leq F_n + 1 \leq 2^n$.
Hence,
\begin{equation}\label{equ:smallprimes}
\log\Big(\prod_{\substack{p^v \,\mid\mid\, \ell_n \\ p \,\leq\, n}} p^v \Big) \leq \log\Big(\prod_{\substack{p^v \,\mid\mid\, \ell_n \\ p \,\leq\, n}} 2^n \Big) \leq \#\{p : p \leq n\} \cdot n \cdot \log 2 = O\!\left(\frac{n^2}{\log n}\right) ,
\end{equation}
since the number of primes not exceeding $x$ is $O(x / \log x)$.

Writing each $F_i$, $L_j$ in~\eqref{equ:logelln2} as a product of $\Phi_d$'s using~\eqref{equ:FnLnPhid}, and taking into account Lemma~\ref{lem:gcdPhi} and~\eqref{equ:smallprimes}, we obtain that

\begin{equation*}
\log \ell_{\mathbf{s}}(n) = \log\prod_{d \,\in\, \mathcal{M}_{\mathbf{s}}(n)} \Phi_d + O\!\left(\frac{n^2}{\log n}\right) .
\end{equation*}
Hence, by Lemma~\ref{lem:logPhi}, we get
\begin{equation*}
\log \ell_{\mathbf{s}}(n) = \sum_{d \,\in\, \mathcal{M}_{\mathbf{s}}(n)} \log \Phi_d + O\!\left(\frac{n^2}{\log n}\right) = \sum_{d \,\in\, \mathcal{M}_{\mathbf{s}}(n)} \varphi(d) \log \alpha + O\!\left(\frac{n^2}{\log n}\right) ,
\end{equation*}
since $\mathcal{M}_{\mathbf{s}}(n) \subseteq [2, n]$ and consequently $\#\mathcal{M}_{\mathbf{s}}(n) \leq n$.
\end{proof}

For all integers $r \geq 0$ and $m \geq 1$, and for every $x \geq 1$, let us define
\begin{equation*}
\mathcal{A}_{r,m}(x) := \{n \leq x : n \equiv r \!\!\!\!\pmod m\} .
\end{equation*}
We need two lemmas about unions of $\mathcal{D}(n)$, respectively $\mathcal{D}^\prime(n)$, with $n \in \mathcal{A}_{r,m}(x)$.

\begin{lem}\label{lem:Dn}
Let $r, m$ be positive integers and let $\mathcal{S}$ be the set of $s \in \{1, \dots, m\}$ such that there exists an integer $t \geq 1$ satisfying $s t \equiv r \pmod m$.
For each $s \in \mathcal{S}$, let $t(s)$ be the minimal $t$.
Then, for all $x \geq 1$, we have
\begin{equation*}
\bigcup_{n \,\in\, \mathcal{A}_{r,m}(x)} \mathcal{D}(n) = \bigcup_{s \,\in\, \mathcal{S}} \mathcal{A}_{s, m}\!\left(\frac{x}{t(s)}\right) .
\end{equation*}
\end{lem}
\begin{proof}
On the one hand, let $n \in \mathcal{A}_{r, m}(x)$ and pick $d \in \mathcal{D}(n)$.
Clearly, $n = dt$ for some integer $t \geq 1$.
Let $s \in \{1, \dots, m\}$ such that $d \equiv s \pmod m$.
Then $s t \equiv d t \equiv n \equiv r \pmod m$ and consequently $s \in \mathcal{S}$ and $t \geq t(s)$.
Therefore, $d = n / t \leq x / t(s)$, so that $d \in \mathcal{A}_{s, m}(x / t(s))$.

On the other hand, suppose that $d \in \mathcal{A}_{s, m}(x / t(s))$ for some $s \in \mathcal{S}$.
Letting $n := dt(s)$, we have $n \equiv s t(s) \equiv r \pmod m$ and $n \leq x$, that is, $n \in \mathcal{A}_{r,m}(x)$.
Finally, $d \in \mathcal{D}(n)$.
\end{proof}

\begin{lem}\label{lem:Dnprime}
Let $r, m$ be positive integers and let $\mathcal{S}$ be the set of $s \in \{1, \dots, m\}$ such that there exists an odd integer $t \geq 1$ satisfying $s t \equiv r \pmod m$.
For each $s \in \mathcal{S}$, let $t(s)$ be the minimal~$t$.
Then, for all $x \geq 1$, we have
\begin{equation*}
\bigcup_{n \,\in\, \mathcal{A}_{r,m}(x)} \mathcal{D}^\prime(n) = \bigcup_{s \,\in\, \mathcal{S}} \mathcal{A}_{2s, 2m}\!\left(\frac{2x}{t(s)}\right) .
\end{equation*}
\end{lem}
\begin{proof}
On the one hand, let $n \in \mathcal{A}_{r, m}(x)$ and pick $d \in \mathcal{D}^\prime(n)$.
Then $2n = dt$ for some odd integer $t \geq 1$.
In particular, $d$ is even.
Let $s \in \{1, \dots, m\}$ such that $\tfrac{d}{2} \equiv s \pmod {m}$.
Then $s t \equiv \tfrac{d}{2} t \equiv n \equiv r \pmod {m}$, and consequently $s \in \mathcal{S}$ and $t \geq t(s)$.
Therefore, $d = 2n / t \leq 2x / t(s)$, so that $d \in \mathcal{A}_{2s, 2m}(2x / t(s))$.

On the other hand, suppose that $d \in \mathcal{A}_{2s, 2m}(2x / t(s))$ for some $s \in \mathcal{S}$.
In particular, $d$ is even and $\tfrac{d}{2} \equiv s \pmod m$.
Letting $n := \tfrac{d}{2}t(s)$, we have $n \equiv s t(s) \equiv r \pmod m$ and $n \leq x$, that is, $n \in \mathcal{A}_{r,m}(x)$.
Finally, $2n = d t(s)$ and $t(s)$ is odd, so that $d \in \mathcal{D}^\prime(n)$.
\end{proof}

Finally, we need two asymptotic formulas for sums of the Euler's function over an arithmetic progression.

\begin{lem}\label{lem:sumtotient}
Let $r,m$ be positive integers.
Then, for every $x \geq 2$, we have
\begin{equation*}
S_{r,m}(x) := \sum_{n \,\in\, \mathcal{A}_{r, m}(x)} \varphi(n) = \frac{3}{\pi^2} \cdot c_{r,m} x^2 + O_{r,m}(x \log x) ,
\end{equation*}
where
\begin{equation*}
c_{r, m} := \frac1{m} \prod_{\substack{p \,\mid\, m \\ p \,\mid\, r}} \left(1 + \frac1{p}\right)^{-1} \prod_{\substack{p \,\mid\, m \\ p \,\nmid\, r}} \left(1 - \frac1{p^2}\right)^{-1} .
\end{equation*}
\end{lem}
\begin{proof}
This is a special case of the asymptotic formula, given by Shapiro~\cite[Theorem~5.5A.2]{MR693458}, for $\sum_{n \leq x} \varphi(f(n))$, where $f$ a polynomial with integers coefficients, no multiple roots, and satisfying $f(n) \geq 1$ for every integer $n \geq 1$.
\end{proof}

\begin{lem}\label{lem:sumtotientexp}
Let $r,m$ be positive integers and let $z \in (0,1)$.
Then, for every $x \geq 2$, we have
\begin{equation*}
\sum_{n \,\in\, \mathcal{A}_{r, m}(x)} \varphi(n)\big(1 - z^{\lfloor x / n \rfloor}\big) = \frac{3}{\pi^2} \cdot \frac{c_{r,m}(1 - z)\Li_2(z)}{z} \cdot x^2 + O_{r, m}\big(x (\log x)^2\big) .
\end{equation*}
\end{lem}
\begin{proof}
For every integer $k \geq 1$, we have $\lfloor x / n \rfloor = k$ if and only if $x / (k + 1) < n \leq x / k$.
Hence,
\begin{align*}
\sum_{n \,\in\, \mathcal{A}_{r, m}(x)} \varphi(n)\big(1 - z^{\lfloor x / n \rfloor}\big) &= \sum_{k \,\leq\, x} \big(1 - z^k\big) \left(S_{r,m}\!\left(\frac{x}{k}\right) - S_{r,m}\!\left(\frac{x}{k + 1}\right)\right) \\
 &= \sum_{k \,\leq\, x} \left(\big(1 - z^k\big) - \big(1 - z^{k-1}\big)\right) S_{r,m}\!\left(\frac{x}{k}\right) \\
 &= (1 - z)\sum_{k \,\leq\, x} z^{k-1}\left(\frac{3}{\pi^2} \cdot \frac{c_{r,m}x^2}{k^2} + O_{r,m}\!\left(\frac{x \log x}{k} \right)\right) \\
 &= \frac{3}{\pi^2} \cdot c_{r,m} (1 - z) \sum_{k \,=\, 1}^\infty \frac{z^{k-1}}{k^2} \cdot x^2 + O_{r,m}\!\left(\sum_{k \,>\, x} \frac{x^2}{k^2} \right) + O_{r,m}\!\left(\sum_{k \,\leq\, x} \frac{x \log x}{k} \right) \\
 &= \frac{3}{\pi^2} \cdot \frac{c_{r,m}(1 - z)\Li_2(z)}{z} \cdot x^2 + O_{r,m}\big(x(\log x)^2\big) ,
\end{align*}
where we employed Lemma~\ref{lem:sumtotient}.
\end{proof}

\section{Proof of Theorem~\ref{thm:periodic}}

Let $\mathbf{s} = (s_n)_{n \geq 1}$ be a periodic sequence in $\{-1, +1\}$, and let $T \geq 1$ be the length of its period.
By the periodicity of $\mathbf{s}$, it follows that there exist $\mathcal{R}_1, \mathcal{R}_2 \subseteq \{1, \dots, m\}$, where $m := 2T$, such that
\begin{equation*}
\mathcal{F}_\mathbf{s}(n) = \bigcup_{r \,\in\, \mathcal{R}_1} \mathcal{A}_{r, m}(n / 2) \quad\text{and}\quad \mathcal{L}_\mathbf{s}(n) = \bigcup_{r \,\in\, \mathcal{R}_2} \mathcal{A}_{r, m}(n / 2) ,
\end{equation*}
for every integer $n \geq 1$.

Then, by Lemma~\ref{lem:Dn} and Lemma~\ref{lem:Dnprime}, we get that there exist $\mathcal{R} \subseteq \{1, \dots, 2m\}$ and positive rational numbers $(\theta_r)_{r \in \mathcal{R}}$ such that
\begin{equation*}
\mathcal{M}_\mathbf{s}(n) = \bigcup_{r \,\in\, \mathcal{R}} \mathcal{A}_{r, 2m}\!\left(\theta_r n\right) ,
\end{equation*}
for every integer $n \geq 1$.

Therefore, Lemma~\ref{lem:lcm} and Lemma~\ref{lem:sumtotient} yield that
\begin{equation*}
\log \ell_\mathbf{s}(n) = \sum_{r \,\in\, \mathcal{R}} \;\sum_{d \,\in\, \mathcal{A}_{r, 2m}(\theta_r n)} \varphi(d) \log \alpha + O\!\left(\frac{n^2}{\log n}\right) = \frac{3 \log \alpha}{\pi^2} \cdot C_{\mathbf{s}} \cdot n^2 + O_\mathbf{s}\!\left(\frac{n^2}{\log n}\right) ,
\end{equation*}
where
\begin{equation*}
C_{\mathbf{s}} := \sum_{r \,\in\, \mathcal{R}} c_{r, 2m} \,\theta_r^2
\end{equation*}
is a positive rational number effectively computable in terms of $s_1, \dots, s_T$.

The proof is complete.

\section{Proof of Theorem~\ref{thm:random}}

Let $\mathbf{s} = (s_n)_{n \geq 1}$ be a sequence of independent and uniformly distributed random variables in $\{-1, +1\}$, and let $n \geq 1$ be a sufficiently large integer.
For every integer $k \in [2, n / 2]$, we have that the event $k \notin \mathcal{F}_\mathbf{s}(n)$, respectively $k \notin \mathcal{L}_\mathbf{s}(n)$, depends only on $s_{2k-2}, s_{2k-1}, s_{2k+1}, s_{2k+2}$.
In~particular, if the integers $k_1, k_2 \in [2, n / 2]$ satisfy $|k_1 - k_2| \geq 3$ then $\big(k_1 \notin \mathcal{F}_\mathbf{s}(n), k_2 \notin \mathcal{F}_\mathbf{s}(n)\big)$, $\big(k_1 \notin \mathcal{L}_\mathbf{s}(n), k_2 \notin \mathcal{L}_\mathbf{s}(n)\big)$, and $\big(k_1 \notin \mathcal{F}_\mathbf{s}(n), k_2 \notin \mathcal{L}_\mathbf{s}(n)\big)$ are pairs of independent events.
Moreover, we have
\begin{equation*}
\mathbb{P}\big[k \notin \mathcal{F}_\mathbf{s}(n) \big] = \mathbb{P}\big[k \notin \mathcal{L}_\mathbf{s}(n) \big] = 2^{-4} = 16^{-1} .
\end{equation*}

Therefore, it follows that
\begin{align*}
\mathbb{P}\big[d \notin \mathcal{M}_\mathbf{s}(n)\big] &= \mathbb{P}\Big[\bigwedge_{\substack{2 \,\leq\, k \,\leq\, n / 2 \\ d \,\in\, \mathcal{D}(k)}} \big(k \notin \mathcal{F}_\mathbf{s}(n)\big) \;\land\; \bigwedge_{\substack{2 \,\leq\, h \,\leq\, n / 2 \\ d \,\in\, \mathcal{D}^\prime(h)}} \big(h \notin \mathcal{L}_\mathbf{s}(n)\big)\Big] \\
&= \prod_{\substack{2 \,\leq\, k \,\leq\, n / 2 \\ d \,\in\, \mathcal{D}(k)}} \mathbb{P}\big[k \notin \mathcal{F}_\mathbf{s}(n)\big] \;\cdot\; \prod_{\substack{2 \,\leq\, h \,\leq\, n / 2 \\ d \,\in\, \mathcal{D}^\prime(h)}} \mathbb{P}\big[h \notin \mathcal{L}_\mathbf{s}(n)\big] \\
&= 16^{-\lfloor n / (2d) \rfloor} \;\cdot\; 16^{-\big(\lfloor n \gcd(2, d) / (2d) \rfloor -\lfloor n / (2d) \rfloor\big)} \\
&= 16^{-\lfloor n \gcd(2, d) / (2d) \rfloor} ,
\end{align*}
for every integer $d \geq 6$.
Consequently, by Lemma~\ref{lem:lcm}, we get
\begin{align}\label{equ:last1}
\mathbb{E}\big[\log \ell_\mathbf{s}(n)\big] &= (\log \alpha) \sum_{d \,\in\, \mathcal{M}_\mathbf{s}(n)} \varphi(d) \,\mathbb{P}[d \in \mathcal{M}_\mathbf{s}(n)] + O\!\left(\frac{n^2}{\log n}\right) \\
 &= (\log \alpha) \sum_{d \,\leq\, n} \varphi(d) \left(1 - 16^{-\lfloor n \gcd(2, d) / (2d) \rfloor}\right)  + O\!\left(\frac{n^2}{\log n}\right) . \nonumber
\end{align}
In turn, by Lemma~\ref{lem:sumtotientexp}, we have
\begin{align}\label{equ:last2}
\sum_{d \,\leq\, n} \varphi(d) \left(1 - 16^{-\lfloor n \gcd(2, d) / (2d) \rfloor}\right) &= \sum_{d \,\in\, \mathcal{A}_{1,2}(n / 2)} \varphi(d) \left(1 - 16^{-\lfloor n / (2d) \rfloor}\right) \\
&\phantom{MMMMM}+ \sum_{d \,\in\, \mathcal{A}_{2,2}(n)} \varphi(d) \left(1 - 16^{-\lfloor n / d \rfloor}\right) \nonumber\\
&= \frac{3}{\pi^2} \cdot \left(\frac{c_{1,2}}{4} + c_{2,2}\right) 15 \Li_2(1 / 16) \cdot n^2 + O\big(n (\log n)^2\big) \nonumber\\
&= \frac{3}{\pi^2} \cdot \frac{15 \Li_2(1 / 16)}{2} \cdot n^2 + O\big(n (\log n)^2\big) . \nonumber
\end{align}
Finally, putting together~\eqref{equ:last1} and~\eqref{equ:last2}, we obtain
\begin{equation*}
\mathbb{E}\big[\log \ell_\mathbf{s}(n)\big] \sim \frac{3 \log \alpha}{\pi^2} \cdot \frac{15 \Li_2(1 / 16)}{2} \cdot n^2 ,
\end{equation*}
as $n \to +\infty$.

The proof is complete.

\bibliographystyle{amsplain}

\begin{thebibliography}{10}

\bibitem{MR1077711}
S.~Akiyama, \emph{Lehmer numbers and an asymptotic formula for {$\pi$}}, J.
  Number Theory \textbf{36} (1990), no.~3, 328--331.

\bibitem{MR1242715}
S.~Akiyama, \emph{A new type of inclusion exclusion principle for sequences and
  asymptotic formulas for {$\zeta(k)$}}, J. Number Theory \textbf{45} (1993),
  no.~2, 200--214.

\bibitem{MR1394375}
S.~Akiyama, \emph{A criterion to estimate the least common multiple of sequences
  and asymptotic formulas for {$\zeta(3)$} arising from recurrence relation of
  an elliptic function}, Japan. J. Math. (N.S.) \textbf{22} (1996), no.~1,
  129--146.

\bibitem{MR3150887}
S.~Akiyama and F.~Luca, \emph{On the least common multiple of {L}ucas
  subsequences}, Acta Arith. \textbf{161} (2013), no.~4, 327--349.

\bibitem{MR2399185}
Y.~Bugeaud, F.~Luca, M.~Mignotte, and S.~Siksek, \emph{Fibonacci numbers at
  most one away from a perfect power}, Elem. Math. \textbf{63} (2008), no.~2,
  65--75.

\bibitem{MR993902}
P.~Kiss and F.~M\'{a}ty\'{a}s, \emph{An asymptotic formula for {$\pi$}}, J.
  Number Theory \textbf{31} (1989), no.~3, 255--259.

\bibitem{MR2892008}
D.~Marques, \emph{The order of appearance of integers at most one away from
  {F}ibonacci numbers}, Fibonacci Quart. \textbf{50} (2012), no.~1, 36--43.

\bibitem{MR1712797}
Y.~V. Matiyasevich and R.~K. Guy, \emph{A new formula for {$\pi$}}, Amer. Math.
  Monthly \textbf{93} (1986), no.~8, 631--635.

\bibitem{MR1089516}
W.~L. McDaniel, \emph{The g.c.d. in {L}ucas sequences and {L}ehmer number
  sequences}, Fibonacci Quart. \textbf{29} (1991), no.~1, 24--29.

\bibitem{MR3620575}
P.~Pongsriiam, \emph{Fibonacci and {L}ucas numbers which are one away from
  their products}, Fibonacci Quart. \textbf{55} (2017), no.~1, 29--40.

\bibitem{MR4003803}
C.~Sanna, \emph{Practical numbers in {L}ucas sequences}, Quaest. Math.
  \textbf{42} (2019), no.~7, 977--983.

\bibitem{MR693458}
H.~N. Shapiro, \emph{Introduction to the theory of numbers}, Pure and Applied
  Mathematics, John Wiley \& Sons, Inc., New York, 1983, A Wiley-Interscience
  Publication.

\bibitem{MR491445}
C.~L. Stewart, \emph{On divisors of {F}ermat, {F}ibonacci, {L}ucas, and
  {L}ehmer numbers}, Proc. London Math. Soc. (3) \textbf{35} (1977), no.~3,
  425--447.

\bibitem{MR1114366}
B.~Tropak, \emph{Some asymptotic properties of {L}ucas numbers}, Proceedings of
  the {R}egional {M}athematical {C}onference ({K}alsk, 1988), Pedagog. Univ.
  Zielona G\'{o}ra, Zielona G\'{o}ra, 1990, pp.~49--55.
\end{thebibliography}

\end{document}